\documentclass[a4paper,11pt]{article}
\usepackage[latin1]{inputenc}
\usepackage{graphicx}
\usepackage{amsmath}
\usepackage{amsthm}
\usepackage{amssymb}
\usepackage{amsfonts}
\usepackage{stmaryrd}
\newtheorem{thm}{Theorem}

\newtheorem{la}[thm]{Lemma}

\newtheorem{prop}[thm]{Proposition}
\newtheorem{obs}[thm]{Observation}

\begin{document}

\begin{center}
{\Large Star Decompositions of a Cyclic Polygon}
\end{center}
\begin{center}
{\large Tomoki Nakamigawa
}
\end{center}
\begin{center}
{
Department of Information Science \\%
Shonan Institute of Technology \\%
1-1-25 Tsujido-Nishikaigan, Fujisawa, \\%
Kanagawa 251-8511, Japan \\%
e-mail: {nakamigwt@gmail.com}
}
\end{center}

\begin{abstract}
Let $V$ be a set of vertices on a circumference in the plane.
Let $E$ be a set of directed line segments linking two vertices of $V$.
If $E$ forms a set of closed cycles and for all two adjacent edges $uv$ and $vw$, the vertices $u$, $v$, $w$ are arranged in anti-clockwise order, we call $P(V,E)$ a cyclic polygon.
A star decomposition $\mathcal{S}$ of a cyclic polygon $P$ is a set of star polygons partitioning the region of $P$ with some additional diagonals.
A star decomposition $\mathcal{S}$ is called maximal if there is no other star decomposition $\mathcal{S}'$ such that a set of diagonals of $\mathcal{S}$ is a proper subset of that of $\mathcal{S}'$.

In this paper, it is shown that for any two maximal star decompositions $\mathcal{S}_1$ and $\mathcal{S}_2$ of a common cyclic polygon, $\mathcal{S}_1$ can be transformed into  $\mathcal{S}_2$ by a finite sequence of diagonal flips.
It is also shown that if a cyclic polygon $P$ admits a star decomposition, the number of diagonals contained in a maximal star decomposition of $P$ is $p - (n-2r)(n-2r-1)/2$, where $p$ is the number of all possible diagonals of $P$, $n$ is the number of vertices of $P$, and $r$ is the rotation number of $P$.  
\end{abstract}
keywords: cyclic polygon, star polygon, star decomposition, generalized triangulation, rotation number
\medskip

\newpage

\section{Introduction}
A {\it geometric digraph} $D = D(V, E)$ is a directed graph such that the vertex set $V$ is a point set in the plane, and the edge set $E \subset V \times V$ is a set of directed line segments $uv$ with $u, v \in V$.
For an edge $e = uv$, $u$ is called the {\it head} of $e$, and $v$ is called the {\it tail} of $e$.
An {\it endvertex} of $e$ is a head or a tail of $e$.
For a vertex $v \in V$, the number of edges having $v$ as their heads, or their tails, is called {\it indegree} of $v$, or {\it outdegree} of $v$, respectively. 
Throughout the paper, the term "edge" is used to refer to a directed edge.

\begin{figure}[h]
\centering
\includegraphics[scale=0.3]{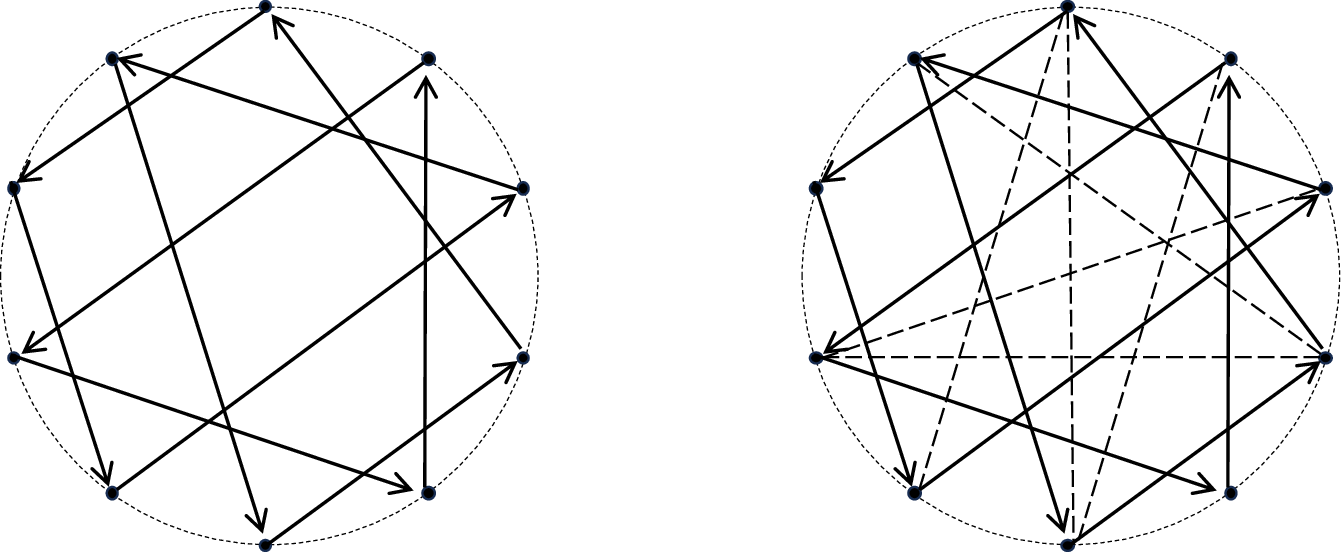}\\
\caption{A cyclic polygon(left) and one of its maximal star decompositions with a set of additional diagonals(right).
}\label{fig_pseudo_convex_polygon}
\end{figure}

Let $V$ be the set of vertices lying on a circumference of the plane.
Let $E$ be the set of edges such that both indgree and outdegree of $v$ are $1$ for all $v \in V$,
and if $uv, vw \in E$, the vertices $u, v, w$ are arranged in anti-clockwise order. 
We call $P = P(V, E)$ a {\it cyclic polygon}
(Fig. \ref{fig_pseudo_convex_polygon}).

For a vertex $v$ of a cyclic polygon $P$, there exists a unique pair of vertices $u$ and $w$ such that $uv, vw$ are edges of $P$.
We denote $u$ and $w$ by $v^{-}$ and $v^{+}$, respectively. 
For $u, v \in V$, we denote the set of vertices $x \in V$ such that $u, x, v$ are arranged in anti-clockwise order by $\llparenthesis uv \rrparenthesis$.
Similarly, we denote the sets $\llparenthesis uv \rrparenthesis \cup \{ u \}$, $\llparenthesis uv \rrparenthesis \cup \{ v \}$, $\llparenthesis uv \rrparenthesis \cup \{ u, v \}$, 
by $\llbracket uv \rrparenthesis$, $\llparenthesis uv \rrbracket$, and $\llbracket uv \rrbracket$, respectively.
For an edge $uv$, the {\it length} of $uv$ is defined as the number of vertices of $V$ contained in $\llparenthesis uv \rrbracket$. 

\begin{figure}[h]
\centering
\includegraphics[scale=0.3]{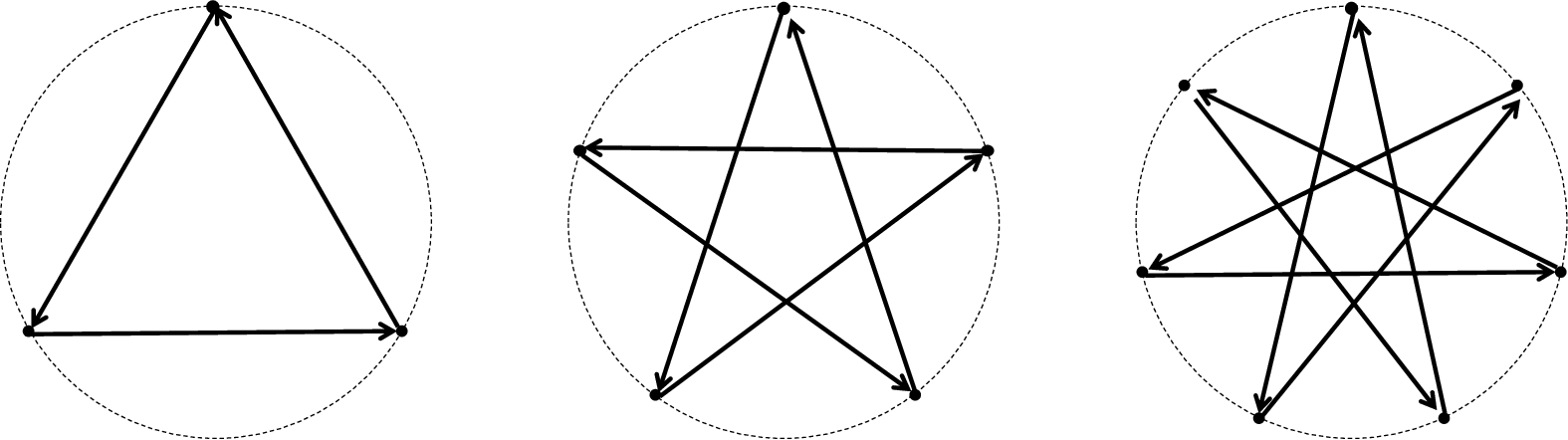}\\
\caption{$k$-stars for
$k=1$(left), $k=2$(center), $k=3$(right).}
\label{fig_star_polygon}
\end{figure}

For a positive integer $k$, a {\it $k$-star polygon}, or simply a {\it $k$-star}, is a cyclic polygon with $2k+1$ vertices such that length of all the  edges is $k$ (Fig. \ref{fig_star_polygon}).

A star is a $k$-star with some positive integer $k$.
We say that two edges $e$ and $f$ {\it share a common point} if $e$ and $f$ have a common endvertex or $e$ and $f$ intersect within the circle.
Remark that a cyclic polygon $P$ is a star if and only if for any two edges of $P$, they share a common point.

For a cyclic polygon $P(V, E)$,
a pair of vertices $u$ and $v$ is called {\it linkable}, if an undirected straight line segment $uv$ splits both $\angle u^{-} u u^{+}$ and $\angle v^{-} v v^{+}$ into a pair of smaller angles $\angle u^{-} u v$, $\angle v u u^{+}$ and $\angle v^{-} v u$, $\angle u v v^{+}$ (Fig. \ref{fig_linkable_pairs}).

\begin{figure}[h]
\centering
\includegraphics[scale=0.3]{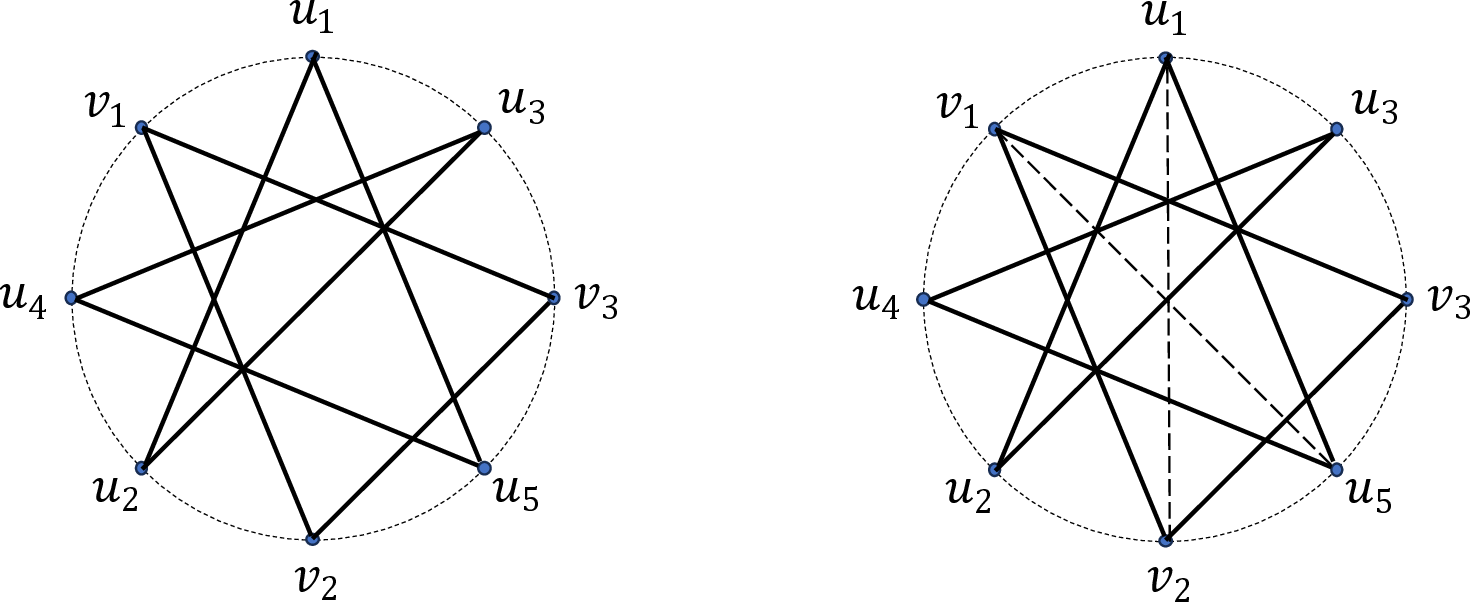}\\
\caption{Two stars $S_1 = u_1 u_2 u_3 u_4 u_5 u_1$ and $S_2 = v_1 v_2 v_3 v_1$ have three linkable pairs $(u_1, v_2)$, $(u_4, v_3)$ and $(u_5, v_1)$ (left).
By a star subdivision using $(u_1, v_2)$ and $(u_5, v_1)$, we have new stars $S'_1 = u_1 v_2 v_3 v_1 u_5 u_1$ and $S'_2 = v_2 u_1 u_2 u_3 u_4 u_5 v_1 v_2$(right).
}\label{fig_linkable_pairs}
\end{figure}

For a linkable pair $u$ and $v$, the undirected line segment $\overline{uv}$ is called a {\it diagonal}.
For a set of diagonals $D$, let us denote the set of all edges $uv$ and $vu$ for $\overline{uv} \in D$ by $\tilde{D}$.  
For a cyclic polygon $P(V, E)$ with a set of diagonals $D$,  a {\it subdivided  cyclic polygon} $P(V, E, D)$
is a geometric graph with the vertex set $V$ and the edge set $E \cup \tilde{D}$.

For a subdivided cyclic polygon $P(V, E, D)$, let us define a function $f_{\rm next}$ from $E \cup \tilde{D}$ to $E \cup \tilde{D}$;
 for $uv \in E \cup \tilde{D}$, $f_{\rm next}(uv)$ is defined as $vw$ such that there is no diagonal of $D$ splitting $\angle uvw$.
Since each diagonal corresponds to two edges of $\tilde{D}$ in opposite directions, $f_{\rm next}$ is well-defined and it turns out to be a bijection.

Since the head of $e$ is the tail of $f_{\rm next} (e)$ for any edge $e$, starting from a given edge $e$, we can traverse a sequence of edges as $e$, $f_{\rm next}(e)$, $f_{\rm next}^2(e)$, $\ldots$, $f_{\rm next}^k(e) = e$ with some $k$ to obtain an oriented cycle, which may have a repeated vertex.
Hence, $E \cup \tilde{D}$ can be decomposed into a set of cycles $E \cup \tilde{D} = C_1 \cup C_2 \cup \cdots \cup C_s$.
We call the decomposition a {\it cycle decomposition} of $P(V, E)$ with respect to $D$.

If all cycles $C_i$'s are stars in a cycle decomposition of $P(V, E)$ with respect to some set of diagonals $D$, the decomposition is called a {\it star decomposition} of $P(V, E)$.
In this situation, we say that $P(V, E)$ admits a star decomposition and we write the decomposition as $\mathcal{S}(V,E, D) = S_1 \cup S_2 \cup \cdots \cup S_s$, where $S_i$'s are stars in the decomposition.
A star decomposition of a cyclic polygon $P$ with respect to $D$ is called {\it maximal} if for any proper superset $D'$ of $D$, there exists no star decomposition of  $P$ with respect to $D'$.

For a convex polygon $P$, a {\it triangulation} of $P$ is a maximal $1$-star decomposition of $P$.
For positive integers $n$ and $k$ with $2k+1 \le n$,
 let $P_n^k$ denote a cyclic polygon with $n$ vertices such that length of all the edges is $k$.
Combinatorial properties of $P_n^k$ have been well investigated.
The following provides an overview of this topic.

\begin{figure}[h]
\centering
\includegraphics[scale=0.3]{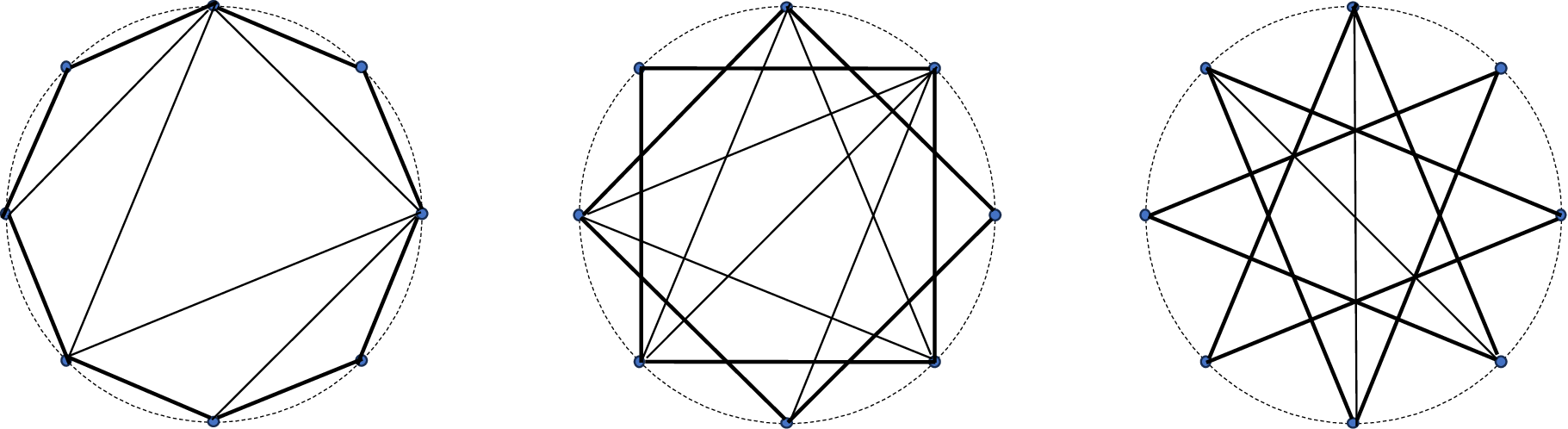}\\
\caption{$k$-triangulations for
$k=1$(left), $k=2$(center), $k=3$(right).
}\label{fig_k_triangulation}
\end{figure}

A pair of diagonals is called a {\it crossing} if they share no endpoint and they intersect within the circle.
A set of $k$ diagonals is called a $k$-{\it crossing} if any pair of them is a crossing.
A set of diagonals $D$ of a cyclic polygon is called $k$-{\it saturated} if $D$ has no $(k+1)$-crossing and $D \cup \{ e \}$ contains a $(k+1)$-crossing with any additional diagonal $e \not\in D$.
For $P_n^k$, if a set of diagonals $D$ is $k$-saturated then $|D| = k(n-2k-1)$ \cite{CP1992, DKM03, Nak00}.
For a positive integer $k$, a star decomposition $\mathcal{S}$ is called a $k$-{\it star decomposition}, or $k$-{\it triangulation}, if all the stars in $\mathcal{S}$ are $k$-stars (Fig. \ref{fig_k_triangulation}). 
Pilaud and Santos considered $P_n^k(V, E, D)$ as a complex of stars and proved that $D$ is $k$-saturated if and only if $P_n^k$ admits a maximal $k$-star decomposition with respect to $D$ \cite{PS2009}. 
Later in Proposition \ref{prop_k_triangulation}, we will remark that any maximal star decomposition of $P_n^k$ is a $k$-star decomposition of $P_n^k$.

A {\it diagonal flip} for a star decomposition $\mathcal{S}(V,E,D)$ is to replace a diagonal $e \in D$ with a new diagonal $f \not\in D$ to build another star decomposition $\mathcal{S'}(V,E,D')$, where $D' = (D \setminus \{ e \}) \cup \{ f \}$.
For a $k$-star decomposition of $P_n^k$ with respect to $D$, for any diagonal $e \in D$, there uniquely exists a diagonal $f \not\in D$ such that there is a $k$-star decomposition of $P_n^k$ with respect to $D' = (D \setminus \{ e \}) \cup \{ f \}$, and any two $k$-star decompositions of $P_n^k$ can be transformed into each other by a finite sequence of diagonal flips \cite{Nak00, PS2009}.  

Jonsson proved that the number of $k$-star decompositions of $P_n^k$ is \\ $\det (C_{n-i-j})_{1 \le i, j \le k}$ \cite{Jon05}.
As facets of a simplicial complex, $k$-star decompositions of $P_n^k$ and related subjects have been extensively studied in terms of pseudo arrangements on sorting networks \cite{PP2012, PS2012}, and subword complexes \cite{PS2015, SS2012, Stump2011}.    
In particular, Serrano and Stump showed an explicit bijection between a family of $k$-star decompositions of $P_n^k$ and a family of $k$-fans of Dick paths of length $2(n-2k)$ \cite{SS2012}.
Refer to Section $3$ of \cite{PSZ2023}, a rich survey of related topics.

The aim of the paper is to build upon a seminal perspective on $k$-triangulations presented by \cite{PS2009} and to study combinatorial properties of star decompositions for a cyclic polygon.
The main results of the paper are the following theorems.

\begin{thm}\label{thm_flip_uniqueness}
Let $\mathcal{S}$ be a maximal star decomposition of a cyclic polygon $P$ with respect to a set of diagonals $D$.
Then for any diagonal $e \in D$, there uniquely exists a diagonal $f \not\in D$ such that there is a maximal star decomposition $\mathcal{S'}$ of $P$ with respect to $D' = (D \setminus \{ e \}) \cup \{ f \}$. 
\end{thm}

The proof of Theorem \ref{thm_flip_uniqueness} will be shown in Section $2$.

\begin{thm}\label{thm_star_decomposition}
Let $\mathcal{S}_1$ and $\mathcal{S}_2$ be a pair of maximal star decompositions of a common cyclic polygon.
Then $\mathcal{S}_1$ can be transformed into $\mathcal{S}_2$ by a finite sequence of diagonal flips. 
\end{thm}

The proof of Theorem \ref{thm_star_decomposition} will be shown in Section $3$.

For a cyclic polygon $P(V, E)$, and for $v \in V$ the {\it exterior angle} of $v$ is $\pi - \angle v^{-} v v^{+}$.
Then the sum of exterior angles over all vertices of $V$ turns out to be a multiple of $2 \pi$.
The {\it rotation number} of $P$, denoted by $r(P)$, is defined as a positive integer $r$ such that the sum of the exterior angles equals $2\pi r$.       

\begin{la}\label{la_num_stars}
Let $\mathcal{S}$ be a star decomposition of a cyclic polygon $P$.
Then $\mathcal{S}$ contains $n - 2r$ stars, where $n$ is the number of vertices of $P$ and $r$ is the rotation number of $P$. 
\end{la}

\begin{proof}
Since the sum of exterior angles of $P$ is $2\pi w$, the sum of internal angles of $P$ is $\pi n - 2\pi w = \pi (n-2r)$.
On the other hand, each star polygon has $\pi$ as the sum of its internal angles.
Hence, we have the number of stars is $n-2r$. 
\end{proof}

For example, for a cyclic polygon $P$ in Fig.\ref{fig_pseudo_convex_polygon}, we have $r(P) = 3$.
Hence, by Lemma \ref{la_num_stars}, there exist $n - 2r = 10 - 2\times 3=4$ stars in any star decomposition of $P$. 

\begin{thm}\label{thm_num_diagonals}
Let $P$ be a cyclic polygon which admits a star decomposition.
Let $n, p, r$ be the number of vertices of $P$, the number of linkable pairs of $P$ and the rotation number of $P$.
Then the number of diagonals contained in a maximal star decomposition of $P$ is $p - (n-2r)(n-2r-1)/2$.
\end{thm}

For example, for a cyclic polygon $P$ in Fig.\ref{fig_pseudo_convex_polygon}, we have $p=12$ linkable pairs.
Hence, by Theorem \ref{thm_num_diagonals}, there exist $p - (n-2r)(n-2r-1)/2 = 12 - (10-2\times 3)(10-2\times 3-1)/2=6$ diagonals in any maximal star decomposition of $P$. 

The proof of Theorem \ref{thm_num_diagonals} will be shown in Section $2$.

The rest of the paper is organized as follows:
In Section $2$, we show basic properties of star polygons, which play key roles to prove the main results.
In Section $3$, we will prove the main theorem, Theorem \ref{thm_star_decomposition}.
In Section $4$, we discuss some open problems.

\section{Basic Properties of Star Polygons}

Two angles $\angle x_i y_i z_i$ for $i=1,2$ are called {\it overlapped} if $y_1=y_2$ and $\llparenthesis z_1 x_1 \rrbracket \cap \llparenthesis z_2 x_2 \rrbracket \ne \emptyset$.
A family of cyclic polygons $\mathcal{P}$ is called {\it independent} if no two angles of distinct polygons of $\mathcal{P}$ are mutually overlapped.
Note that if a cyclic decomposition of a subdivided cyclic polygon yields a star decomposition $\mathcal{S}$, we have $\mathcal{S}$ is independent.

We notice an observation useful for simplifying the proofs of the results described later.

\begin{obs}\label{obs_linkable_observation}
Let $S_1(V_1, E_1)$ and $S_2(V_2, E_2)$ be mutually independent stars. 
If $V_1 \cap V_2 \ne \emptyset$, then we have mutually independent stars $S'_1(V'_1, E'_1)$ and $S'_2(V'_2, E'_2)$ satisfying following conditions.
\begin{enumerate}
\item $V'_1 \cap V'_2 = \emptyset$.
\item $S'_i$ and $S_i$ is isomorphic as directed graphs for $i=1,2$.
\item A pair of edges of $S'_1$ and $S'_2$ is a crossing if and only if its corresponding pair of edges of $S_1$ and $S_2$ is a crossing. 
\item A pair of vertices of $S'_1$ and $S'_2$ is linkable if and only if its corresponding pair of vertices of $S_1$ and $S_2$ is linkable.
\end{enumerate}
\end{obs}

\begin{proof}
Let $v$ be a vertex of $V_1 \cap V_2$ and $u_i v$, $v w_i$ are edges of $S_i$ for $i=1,2$.
Since $S_1$ and $S_2$ are independent, $\angle u_1 v w_1$ and $\angle u_2 v w_2$ are not overlapped.
We may assume $v \in \llparenthesis u_1 u_2 \rrparenthesis$.
Let us split $v$ to two new vertices $v_1$ and $v_2$ such that $u_1, v_1, v, v_2, u_2$ are arranged in anti-clockwise order and $\llparenthesis v_1 v_2 \rrparenthesis$ contains no vertex other than $v$.
Then let us define $V'_i$ as $(V_i \setminus \{ v \}) \cup \{ v_i \}$.
In the same manner, we can split all the vertices contained in $V_1 \cap V_2$, and build a pair of stars $S'_1$ and $S'_2$.
\end{proof}

\begin{la}\label{la_linkable_pairs_of_two_stars}
Let $S_1(V_1, E_1)$ and $S_2(V_2, E_2)$ be mutually independent stars. 
Then there exist odd number of linkable pairs of vertices $(v_1, v_2)$ with $v_i \in V_i$ for $i = 1, 2$.
Furthermore, for distinct linkable pairs, their corresponding diagonals are crossing with each other. 
\end{la}

\begin{proof}
Let $S_1$ and $S_2$ be an $m$-star and an $n$-star, respectively.
By Observation \ref{obs_linkable_observation}, we may assume $V_1 \cap V_2 = \emptyset$.

Note that $v_1 \in V_1$ and $v_2 \in V_2$ are linkable if and only if $\llparenthesis v_1 v_2 \rrparenthesis$ and $\llparenthesis v_2 v_1 \rrparenthesis$ contain the same number of vertices of $V_i \setminus \{ v_i \}$ for $i = 1, 2$.
Let us relabel $V_1 \cup V_2$ as $\{ x_1, x_2, \ldots, x_{2s} \}$, where $s=m+n+1$, such that $x_1, x_2, \ldots, x_{2s}$ are arranged in anti-clockwise order.

In the following, for an integer $i$, let $\overline{i}$ denote the integer such that $\overline{i} \equiv i \pmod{2s}$ and $\overline{i} \in [1, 2s]$.   
Let us define a weight $w$, a function from $[1, 2s]$ to the set of integers;
\begin{eqnarray*}
w(i) & = & |\{ j \in [1, 2s] \,:\, x_j \in V_1 {\rm ~and~} 1 \le \overline{j-i} \le s-1  \}| \\
 & & -  |\{ j \in [1, 2s] \,:\, x_j \in V_1 {\rm ~and~} s+1 \le \overline{j-i} \le 2s-1  \}|
\end{eqnarray*}
for $1 \le i \le 2s$  (Fig. \ref{fig_linkable_weight}).
Note that for $x_i \in V_1$ and $x_j \in V_2$, $x_i$ and $x_j$ are linkable if and only if $\overline{j - i} = s$  and $w(i) = 0$.

\begin{figure}[h]
\centering
\includegraphics[scale=0.3]{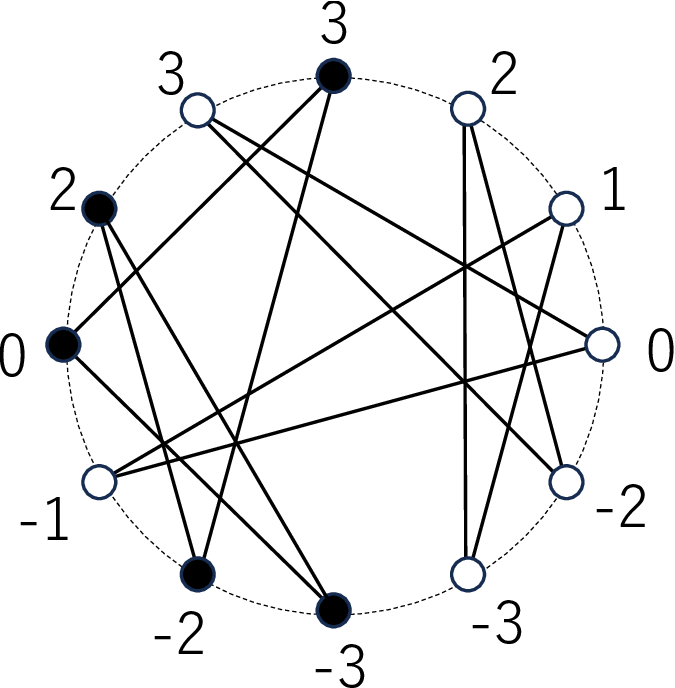}\\
\caption{The vertices $V_1$ and $V_2$ are displayed in black and white, respectively.
Labels are their weights. 
}\label{fig_linkable_weight}
\end{figure}

First, we will show that $w(\alpha) = 0$ for some $\alpha \in [1, 2s]$.

By definition of the weight $w$, we have $w(i) + w(\overline{s+i}) = 0$ for all $i \in [1, 2s]$.
Since both $2m+1$ and $2n+1$ are odd numbers, there exists some $i$ for $1 \le i \le s$ such that one of $x_i$ and $x_{s+i}$ is contained in $V_1$ and the other is contained in $V_2$.
Hence, by replacing the indices if necessary, we may assume $x_1 \in V_1$ and $x_{s+1} \in V_2$.

Let us define families of indices $W_k \subset [1, 2s]$ for $0 \le k \le 2$ as follows;
$W_k = \{ i \in [1, 2s] \,:\, x_i \in V_k, x_{\overline{s+i}} \not\in V_k \}$ for $k=1,2$ and $W_0 = [1,2s] \setminus (W_1 \cup W_2)$.
Then, we have $|W_1| = |W_2|$, and both $|W_1|$ and $|W_2|$ are odd numbers.
Furthermore, we have $w(i)$ is even for $i \in W_1 \cup W_2$, and $w(i)$ is odd for $i \in W_0$.
We will list the difference between $w(\overline{i+1})$ and $w(i)$;
\begin{eqnarray}
 & & w(\overline{i+1}) \notag \\
 & = &
   \left\{
\begin{array}{ll}
w(i) + 2 &{\rm ~for~}(i, \overline{i+1}) \in W_2 \times W_2 \\
w(i) - 2 &{\rm ~for~}(i, \overline{i+1}) \in W_1 \times W_1 \\
w(i) + 1 &{\rm ~for~}(i, \overline{i+1}) \in (W_0 \times W_2) \cup (W_2 \times W_0)\\
w(i) - 1 &{\rm ~for~}(i, \overline{i+1}) \in (W_0 \times W_1) \cup (W_1 \times W_0)\\
w(i) &{\rm ~for~}(i, \overline{i+1}) \in (W_1 \times W_2) \cup (W_2 \times W_1) \cup (W_0 \times W_0).
\end{array}
   \right.
\end{eqnarray}

From (1), for $i$, $j$ with $1 \le i < j \le 2s$, if $\{ i, j \} \subset W_1 \cup W_2$ and $k \in W_0$ for all $i < k < j$, we have
\begin{eqnarray}
w(j) & = & 
   \left\{
\begin{array}{ll}
w(i) + 2 &{\rm ~for~}(i,j) \in W_2 \times W_2 \\
w(i) - 2 &{\rm ~for~}(i,j) \in W_1 \times W_1 \\
w(i) &{\rm ~for~}(i,j) \in (W_1 \times W_2) \cup (W_2 \times W_1).
\end{array}
   \right.
\end{eqnarray} 
 
Since $w(s+1) = -w(1)$, we have some $\alpha$ with $1 \le \alpha \le s$ such that $w(\alpha) = 0$.
We may assume $\alpha \in W_1$.
Let $\beta$ be the smallest $i$ such that $i \in W_1 \cup W_2$ and $w(i) = 0$ with $\alpha < i \le \alpha + s$.
Since $\alpha + s \in W_2$ and $w(\alpha + s) = 0$, $\beta$ is well-defined, and we have $\alpha < \beta \le \alpha + s$.

We claim that $\beta \in W_2$.
Suppose to a contradiction that $\beta \in W_1$.
From (2) and by the minimality of $\beta$, we have $w(i) < 0$ for all $i$ satisfying $\alpha < i < \beta$ and $i \in W_1 \cup W_2$.
It follows that $w(\beta) < 0$, a contradiction.

Let us define $U = \{ i \in W_1 \cup W_2 \,:\,  \alpha \le i < \alpha + s, w(i)=0 \}$ and let $U = \{ \alpha_1, \alpha_2, \ldots, \alpha_t \}$ with $\alpha = \alpha_1 < \alpha_2 < \ldots < \alpha_t < \alpha_{t+1} = \alpha+s$.
Then we have $\alpha_1 \in W_1$ and $\alpha_2 \in W_2$.
In the same manner, we have $\alpha_i \in W_1$ for $i$ odd, and $\alpha_i \in W_2$ for $i$ even.
Since $\alpha_{t+1}$, which is $\alpha + s$, is contained in $W_2$, we have $t$ is odd.

Furthermore, for all linkable pairs $( x_{\alpha_i}, x_{\alpha_i+s} )$ with $1 \le i \le t$, their corresponding diagonals $\overline{x_{\alpha_i} x_{\alpha_i+s}}$'s are crossing with each other.
\end{proof}

The next lemma shows that if a pair of stars has multiple linkable pairs in a star decomposition $\mathcal{S}$, we can build another star decomposition $\mathcal{S}'$ from $\mathcal{S}$ with two additional diagonals.

\begin{la}\label{la_star_subdivision}
Let $S_1(V_1, E_1)$ and $S_2(V_2, E_2)$ be two distinct stars having two linkable pairs $(u_1, u_2)$ and $(v_1, v_2)$.
Then we have a new pair of stars $S'_1(V'_1, E'_1)$ and $S'_2(V'_2, E'_2)$ such that   $V'_1 \cup V'_2 = V_1 \cup V_2$ and $E'_1 \cup E'_2 = E_1 \cup E_2 \cup \{ u_1 u_2, u_2 u_1, v_1 v_2, v_2 v_1 \}$.
\end{la}

\begin{proof}
Let $\angle u_i^{-} u_i u_i^{+}$ and $\angle v_i^{-} v_i v_i^{+}$ be angles of $S_i$ for $i=1,2$.
In $S_i$, the vertex set $V_i$ is traversed along the directions of edges as $u_i u_i^{+} P_i v_i^{-} v_i v_i^{+} Q_i u_i^{-} u_i$ with some vertex sets $P_i, Q_i \subset V_i$ for $i=1, 2$. 

By using additional edges $u_1 u_2, u_2 u_1, v_1 v_2, v_2 v_1$, we have two cycles $S'_1(V'_1, E'_1)$ and $S'_2(V'_2, E'_2)$ as follows;
for $S'_1$, by using $u_1 u_2$ and $v_2 v_1$, $V'_1$ is traversed as $u_1 u_2 u_2^{+} P_2 v_2^{-} v_2 v_1 v_1^{+} Q_1 u_1^{-} u_1$, and
for $S'_2$, by using $u_2 u_1$ and $v_1 v_2$, $V'_2$ is traversed as $u_2 u_1 u_1^{+} P_1 v_1^{-} v_1 v_2 v_2^{+} Q_2 u_2^{-} u_2$ (Fig. \ref{fig_star_subdivision_proof}).

\begin{figure}[h]
\centering
\includegraphics[scale=0.3]{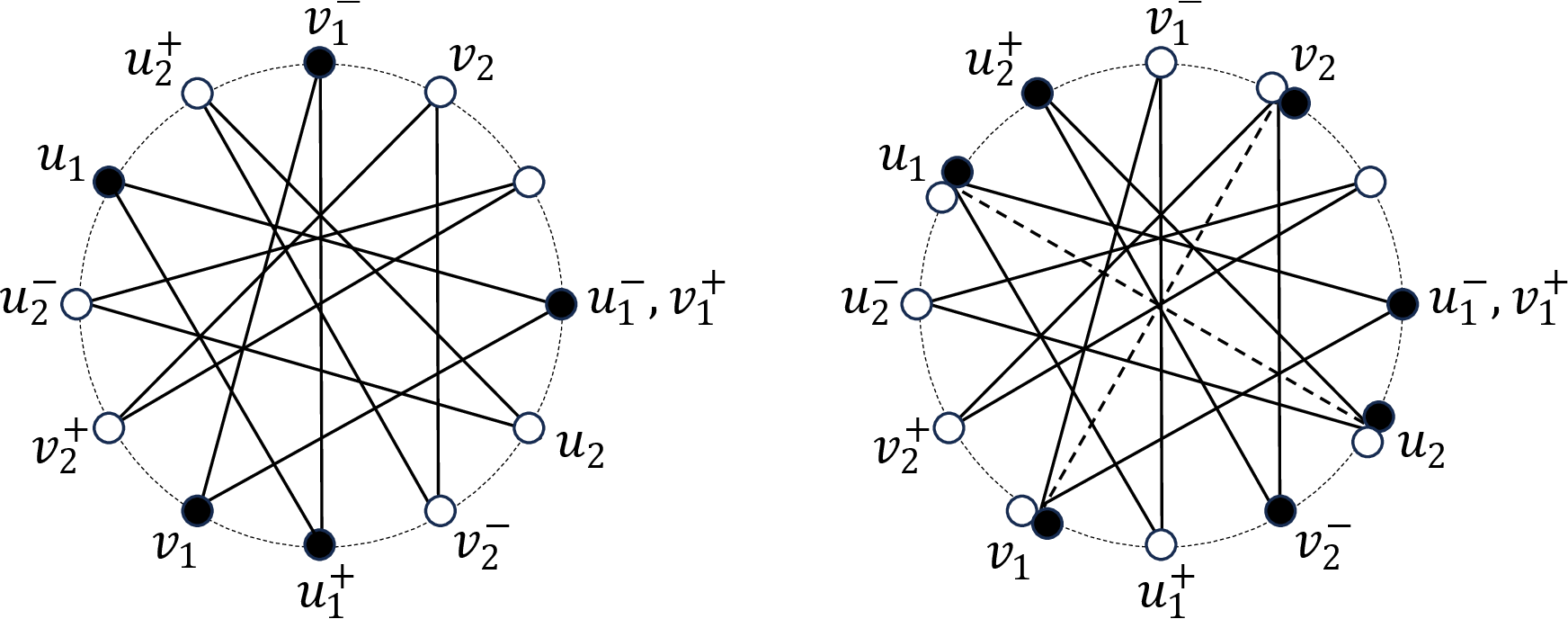}\\
\caption{The vertices of $S_1$ and $S_2$ are displayed in black and white, respectively(left).
By using additional diagonals $\overline{u_1 u_2}$ and $\overline{v_1 v_2}$, we have two stars $S'_1$ and $S'_2$.
The vertices of $S'_1$ and $S'_2$ are displayed in black and white, respectively(right). 
}\label{fig_star_subdivision_proof}
\end{figure}

Next, we will show that $S'_1$ and $S'_2$ are stars.
By symmetry, we only show that $S'_1$ is a star.
It suffices to show that for two edges $e$ and $f$ of $E'_1$, they share a common point.

Firstly, if $e=u_1 u_2$ and $f=v_2 v_1$, they are crossing by Lemma \ref{la_linkable_pairs_of_two_stars}.
Secondly, if $e \in \{ u_1 u_2, v_2 v_1 \}$ and $f \in E_1 \cup E_2$, then $e$ and $f$ share a common point, because for two stars, a diagonal corresponding to a linkable pair and an edge of original two stars always share a common point in general. 
Hence, we may assume $e$ is an edge of  $u_2 u_2^{+} P_2 v_2^{-} v_2$ and $f$ is an edge of $v_1 v_1^{+} Q_1 u_1^{-} u_1$. 

The circumference is partitioned into four parts by the vertices $u_1$, $u_2$, $v_1$ and $v_2$.
Let us define $R_i$ with $1 \le i \le 4$ as follows;
if $v_1 \in \llparenthesis u_1 u_2 \rrparenthesis$, then we have
$R_1 = \llbracket u_1 v_1 \rrbracket$,
$R_2 = \llbracket v_1 u_2 \rrbracket$,
$R_3 = \llbracket u_2 v_2 \rrbracket$,
$R_4 = \llbracket v_2 u_1 \rrbracket$,
otherwise, we have
$R_1 = \llbracket u_1 v_2 \rrbracket$,
$R_2 = \llbracket v_2 u_2 \rrbracket$,
$R_3 = \llbracket u_2 v_1 \rrbracket$,
$R_4 = \llbracket v_1 u_1 \rrbracket$.

Since any edge contained in $E_1 \cup E_2$ are crossing $u_1 u_2$ and $v_2 v_1$, we have $u_2^{+} \in R_4$ and $u_1^{-} \in R_3$.
Then we have $e$ links two vertices of $R_2$ and $R_4$ and $f$ links two vertices of $R_1$ and $R_3$.
Hence, $e$ and $f$ share a common point.
\end{proof}

The reconstruction of two stars by additional two diagonals in Lemma \ref{la_star_subdivision} is called a {\it star subdivision}.
A pair of stars $S_1$ and $S_2$ contained in a star decomposition is called a {\it double star}, if $S_1$ and $S_2$ share at least one common diagonal.
Note that Lemma \ref{la_linkable_pairs_of_two_stars} and Lemma \ref{la_star_subdivision} can be applied for a double star. 

By applying Lemma \ref{la_linkable_pairs_of_two_stars} and Lemma \ref{la_star_subdivision}, we will prove Theorem \ref{thm_num_diagonals}.

\begin{proof}[Proof of Theorem~{\upshape\ref{thm_num_diagonals}}]
For a maximal star decomposition $\mathcal{S}$ of $P$, let $s$ be the number of stars of $\mathcal{S}$.
By Lemma \ref{la_num_stars}, we have $s = n-2r$.
For a pair of stars in $\mathcal{S}$, by Lemma \ref{la_linkable_pairs_of_two_stars}, we have an odd number of linkable pairs of vertices.
If there exists a pair of stars having at least $3$ linkable pairs, by Lemma \ref{la_star_subdivision}, we have $\mathcal{S}$ is not maximal, a contradiction.
Therefore, there remain exactly $\binom{s}{2}$ linkable pairs in $\mathcal{S}$.
Hence, the number of diagonals is $p - \binom{s}{2} = p - (n-2r)(n-2r-1)/2$. 
\end{proof}



The next lemma shows that for a double star, any diagonal contained in both stars can be replaced with another diagonal to yield a new double star (Fig. \ref{fig_diagonal_flip}).

\begin{figure}[h]
\centering
\includegraphics[scale=0.3]{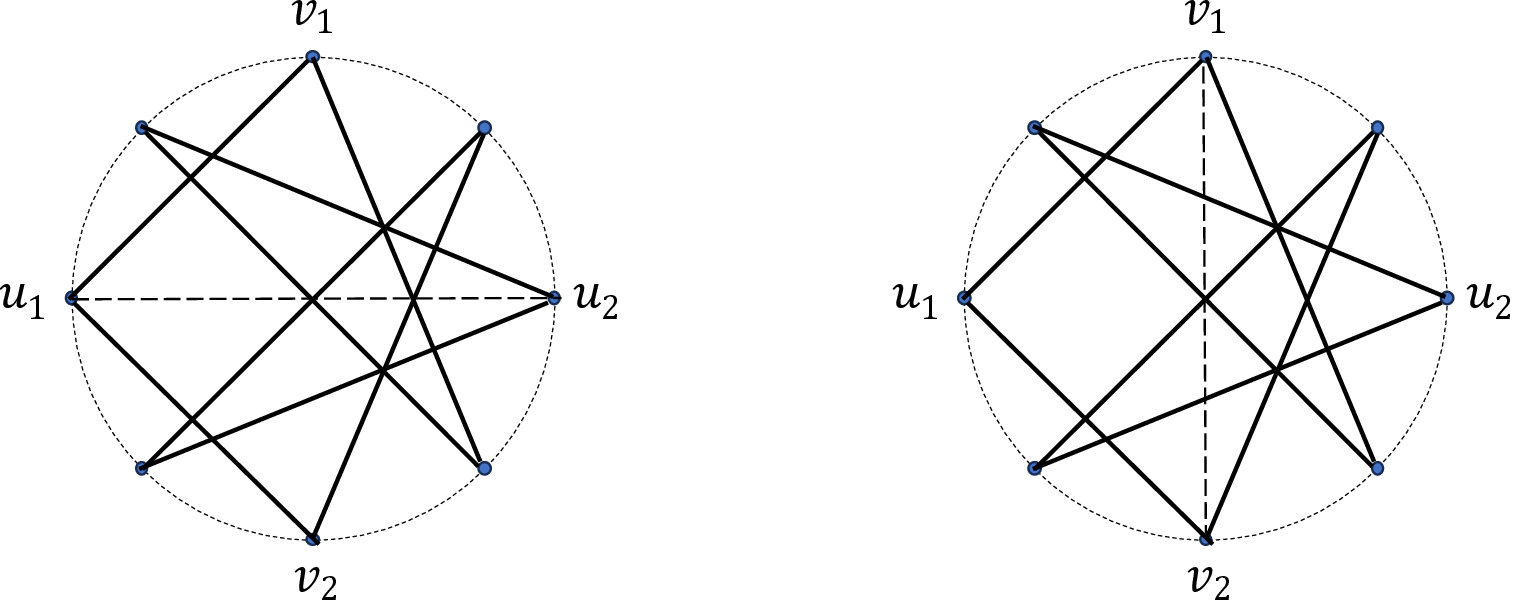}\\
\caption{A diagonal flip to replace $\overline{u_1 u_2}$ with $\overline{v_1 v_2}$.
}\label{fig_diagonal_flip}
\end{figure}

\begin{la}\label{la_diagonal_flip}
Let $\mathcal{S}$ be a double star with two stars $S_1(V_1, E_1)$ and $S_2(V_2, E_2)$ such that $e=\overline{u_1 u_2}$ is a common diagonal and $f=(v_1, v_2)$ is a linkable pair of  $S_1$ and $S_2$.
Then there exists a double star $\mathcal{S}'$ with two stars $S'_1(V'_1, E'_1)$ and $S'_2(V'_2, E'_2)$ with a common diagonal $f$ such that $V'_1 \cup V'_2 = V_1 \cup V_2$ and $E'_1 \cup E'_2 = ((E_1 \cup E_2) \setminus \{ u_1 u_2, u_2 u_1 \}) \cup \{ v_1 v_2, v_2 v_1 \}$.
\end{la}

\begin{proof}
Let $\angle u_1^{-} u_1 u_2$, $\angle u_1 u_2 u_2^{+}$ and $\angle v_1^{-} v_1 v_1^{+}$ be angles of $S_1$.
In $S_1$, the vertex set $V_1$ is traversed as $u_1 u_2 u_2^{+} P_1 v_1^{-} v_1 v_1^{+} Q_1 u_1^{-} u_1$ with some vertex sets $P_1, Q_1 \subset V_1$.
Let $\angle u_2^{-} u_2 u_1$, $\angle u_2 u_1 u_1^{+}$ and $\angle v_2^{-} v_2 v_2^{+}$ be angles of $S_2$.
In $S_2$, the vertex set $V_2$ is traversed as $u_2 u_1 u_1^{+} P_2 v_2^{-} v_2 v_2^{+} Q_2 u_2^{-} u_2$ with some vertex sets $P_2, Q_2 \subset V_2$.

Let us remove $e$ and add $f$ to $\mathcal{S}$.
Then, we have two cycles $S'_1(V'_1, E'_1)$ and $S'_2(V'_2, E'_2)$ as follows;
for $S'_1$, by using an additional edge $v_1 v_2$, $V'_1$ is traversed as $v_1 v_2 v_2^{+} Q_2 u_2^{-} u_2 u_2^{+} P_1 v_1^{-} v_1$, and
for $S'_2$, by using an additional edge $v_2 v_1$, $V'_2$ is traversed as $v_2 v_1 v_1^{+} Q_1 u_1^{-} u_1 u_1^{+} P_2 v_2^{-} v_2$ (Fig. \ref{fig_diagonal_flip_proof}).

\begin{figure}[h]
\centering
\includegraphics[scale=0.3]{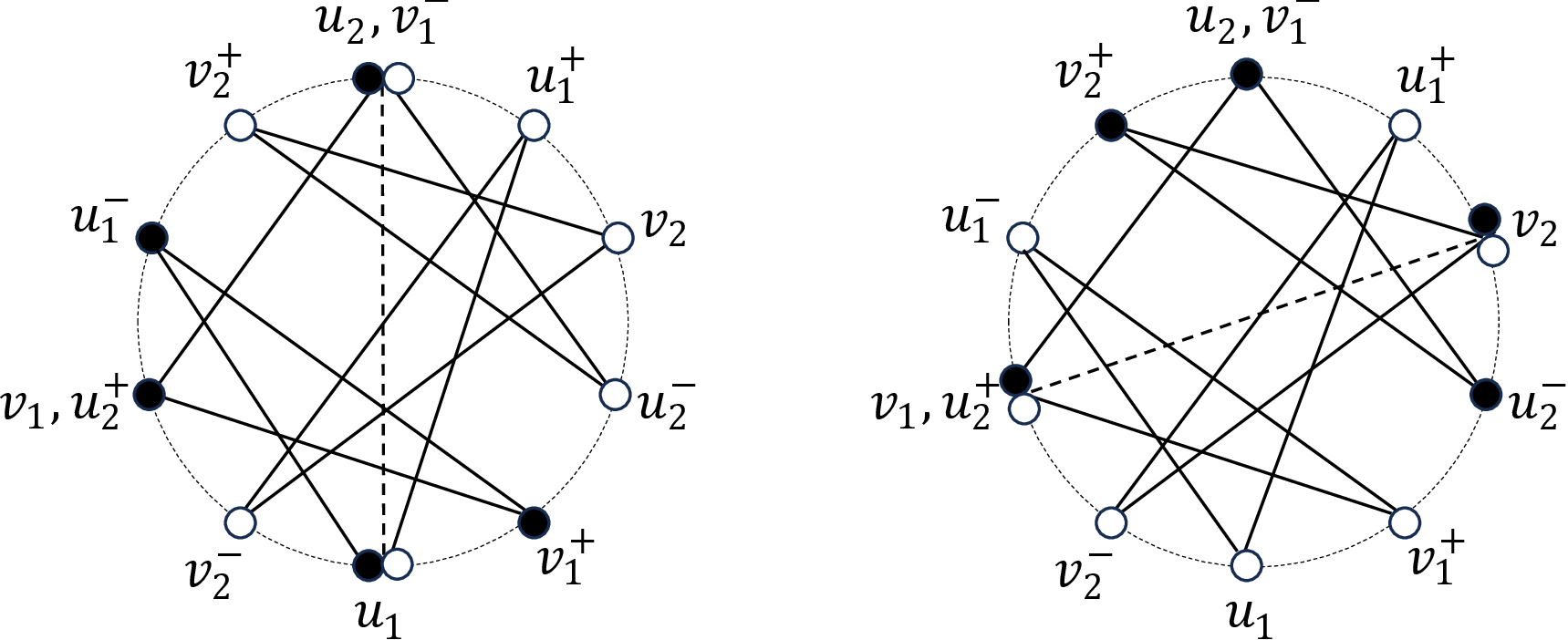}\\
\caption{The vertices of $S_1$ and $S_2$ are displayed in black and white, respectively(left).
By flipping $\overline{u_1 u_2}$ to $\overline{v_1 v_2}$, we have a new double star with $S'_1$ and $S'_2$.
The vertices of $S'_1$ and $S'_2$ are displayed in black and white, respectively(right). 
}\label{fig_diagonal_flip_proof}
\end{figure}

What we want to show is that $S'_1$ and $S'_2$ are stars.
By symmetry, it suffices to show that $S'_1$ is a star.
For two edges $g_1$ and $g_2$ of $S'_1$, we will show that $g_1$ and $g_2$ share a common point.
If $g_1 = v_1 v_2$, since it corresponds to a linkable pair of $S_1$ and $S_2$, we have that $g_1$ and any edge of $S_1$ and $S_2$ share a common point.
Hence, we may assume $g_1$ is an edge of $v_2 v_2^{+} Q_2 u_2^{-} u_2$ and $g_2$ is an edge of $u_2 u_2^{+} P_1 v_1^{-} v_1$.
The circumference is partitioned into four parts by the vertices $u_1$, $u_2$, $v_1$ and $v_2$.
Let us define $R_i$ with $1 \le i \le 4$ as follows;
if $v_1 \in \llparenthesis u_1 u_2 \rrparenthesis$, then we have
$R_1 = \llbracket u_1 v_1 \rrbracket$,
$R_2 = \llbracket v_1 u_2 \rrbracket$,
$R_3 = \llbracket u_2 v_2 \rrbracket$,
$R_4 = \llbracket v_2 u_1 \rrbracket$,
otherwise, we have
$R_1 = \llbracket u_1 v_2 \rrbracket$,
$R_2 = \llbracket v_2 u_2 \rrbracket$,
$R_3 = \llbracket u_2 v_1 \rrbracket$,
$R_4 = \llbracket v_1 u_1 \rrbracket$.

In any case, we have $u_2^{-} \in R_1$.
Then we have $g_1$ links $R_1$ and $R_3$.
Furthermore, we have $u_2^{+} \in R_4$.
Then we have $g_2$ links $R_2$ and $R_4$.
Hence, $g_1$ and $g_2$ share a common point, as claimed.
\end{proof}

By applying Lemma \ref{la_diagonal_flip}, we will prove Theorem \ref{thm_flip_uniqueness}.

\begin{proof}[Proof of Theorem~{\upshape\ref{thm_flip_uniqueness}}]
There uniquely exists a double star with two stars $S_1$ and $S_2$ having $e$ as a common diagonal.
Put $e = \overline{u_1 u_2}$.
Since $(E(S_1) \cup E(S_2)) \setminus \{ u_1 u_2, u_2 u_1 \}$ forms a cycle $C$ other than a star, an additional diagonal $f$ needs to split two angles of $C$.
Since $S_i$ itself contains no linkable pair for $i=1, 2$, $f$ corresponds to a linkable pair of $S_1$ and $S_2$, which uniquely exists by Lemma \ref{la_linkable_pairs_of_two_stars}, since $\mathcal{S}$ is maximal.
Therefore, by Lemma \ref{la_diagonal_flip}, we have a star decomposition $\mathcal{S}'$ with respect to $(D \setminus \{ e \}) \cup \{ f \}$.
\end{proof}

\section{Proof of Theorem \ref{thm_star_decomposition}}

Let $P(V, E)$ be a cyclic polygon.
For $v \in V$, let $\mathcal{S}_v$ be defined as the set of all stars $S$ with respect to some set of diagonals such that $v^{-} v$ is an edge of $S$.
Let $S \in \mathcal{S}_v$ such that $E(S) = \{ v_0 v_1, v_1 v_2, \ldots, v_{2k-1} v_{2k}, v_{2k} v_0 \}$, where $v_1 = v$.
$S$ is called a {\it border star} with respect to $P$ and $v$, if $v_{2i-1} v_{2i}$'s are edges of $P$ for all $1 \le i \le k$, and with respect to this condition, a sequence $(\angle v_1 v_2 v_3, \angle v_3 v_4 v_5, \ldots, \angle v_{2k-1} v_{2k} v_0)$ is the minimum in the lexicographic order (Fig. \ref{fig_border_star}).
Let us denote the border star with respect to $P$ and $v$ by $S^{\ast}(P,v)$, if such a star exists.

\begin{figure}[h]
\centering
\includegraphics[scale=0.3]{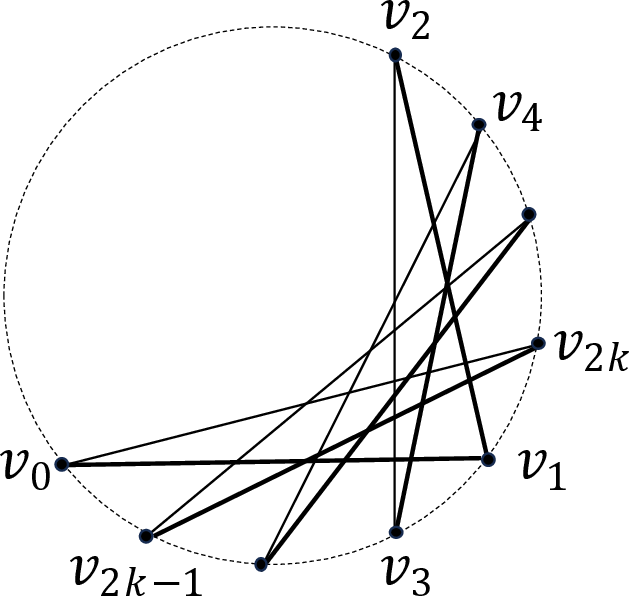}\\
\caption{A border star $S^{\ast}(P,v_1)$ with respect to an underlying cyclic polygon $P$ and a vertex $v_1$.
Thick lines are edges of $P$. 
}\label{fig_border_star}
\end{figure}

For an edge $uv \in E$ of a cyclic polygon $P(V, E)$, $uv$ is called {\it minimal} if there is no edge $xy \in E$ such that $\llbracket xy \rrbracket \subset \llparenthesis uv \rrparenthesis$. 
  
\begin{prop}\label{prop_yield_border_star}
Let $\mathcal{S}$ be a maximal star decomposition of a cyclic polygon $P(V, E)$.
Let $uv$ be a minimal edge of $E$.
Then $\mathcal{S}$ can be transformed into a maximal star decomposition $\mathcal{S}^{\ast}$ by a finite sequence of diagonal flips, where $\mathcal{S}^{\ast}$  contains the border star $S^{\ast}(P,v)$.
\end{prop}

In order to prove Proposition \ref{prop_yield_border_star}, we will use the following lemma.

\begin{la}\label{la_find_linkable_pair}
Let $S_A$ and $S_B$ be two distinct stars in a common star decomposition of a cyclic polygon $P$ with a minimal edge $xy$.
If $a_0 a_1$ and $xy$ are edges of $S_A$ and $b_0 b_1$ is an edge of $S_B$ satisfying that $b_0, y, b_1, a_0, x$ are arranged in anti-clockwise order and $a_1 \in \llbracket x b_0 \rrbracket$,
then there exists a linkable pair $(a, b) \in V(S_A) \times V(S_B)$ such that $a \in \llbracket x a_1 \rrbracket$.
\end{la}

\begin{proof}
By Observation \ref{obs_linkable_observation}, we may assume that $V(S_A) \cap V(S_B) = \emptyset$ and $b_0 \in \llparenthesis a_1 y \rrparenthesis$.
Let $a_0 a_1$, $a_1 a_2$, $\ldots$, $a_{2k-1} a_{2k}$ be edges of $S_A$, where $a_{2k-1} = x$ and $a_{2k} = y$.
For $S_B$, since $xy$ is minimal, we have edges $z_0 z_1$, $z_1 z_2$ of $S_B$ such that $z_0 \in \llparenthesis xy \rrparenthesis$ and $\{ z_1, z_2 \} \subset \llparenthesis yx \rrparenthesis$.
Let $b_0 b_1$, $b_1 b_2$, $\ldots$, $b_{2\ell-1} b_{2\ell}$ be edges of $S_B$, where $b_{2\ell-2} = z_0$, $b_{2\ell-1} = z_1$, $b_{2\ell} = z_2$ (Fig. \ref{fig_find_linkable_pair}).

\begin{figure}[h]
\centering
\includegraphics[scale=0.3]{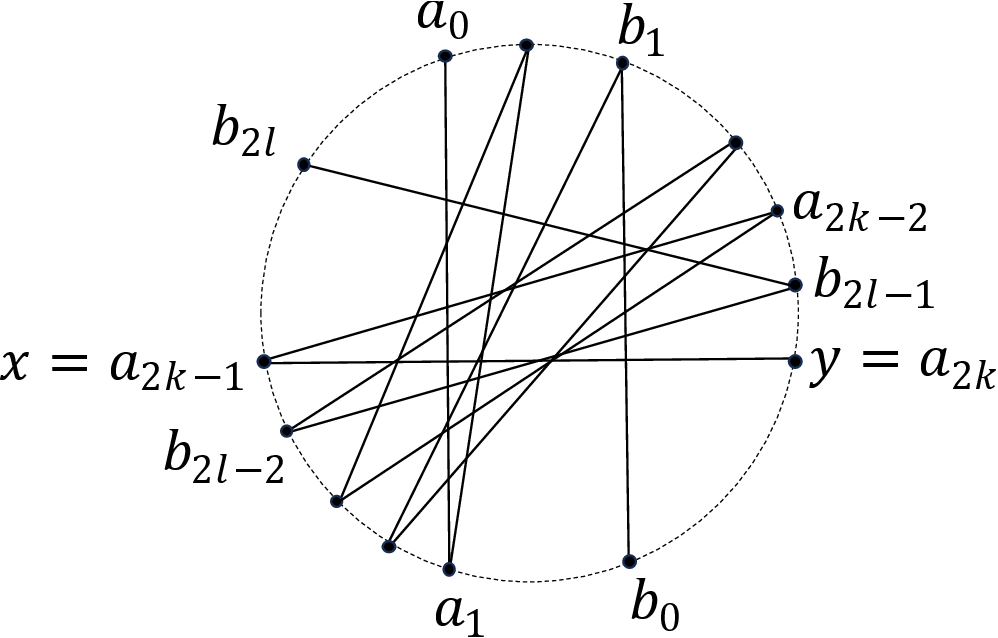}\\
\caption{Vertices and edges of Lemma \ref{la_find_linkable_pair}.}
\label{fig_find_linkable_pair}
\end{figure}

Then we have $a_1, a_3, \ldots, a_{2k-1}=x$ are arranged in clockwise order in $\llbracket xa_1 \rrbracket$, 
and $a_0, a_2, \ldots, a_{2k}=y$ are arranged in clockwise order in $\llbracket ya_0 \rrbracket$.
In the same way, we have $b_1, b_3, \ldots, b_{2\ell-1}$ are arranged in clockwise order in $\llparenthesis yb_1 \rrbracket$, 
and $b_0, b_2, \ldots, b_{2\ell-2}$ are arranged in clockwise order in $\llparenthesis xb_0 \rrbracket$.

Suppose to a contradiction that there exists no linkable pair $(a, b) \in V(S_A) \times V(S_B)$ such that $a = a_{2i-1}$ with some $i$ for $1 \le i \le k$.
\medskip\\
{\bf Claim.} %
For all $j$ with $0 \le j \le \ell-1$, there exists some $i$ with $0 \le i \le k-1$ such that $\{ a_{2i}, a_{2i+1} \} \subset \llparenthesis b_{2j+1} b_{2j} \rrparenthesis$.

\begin{proof}
We proceed by induction on $j$ to prove the claim.
For $j=0$, by the assumption of the lemma, we have $\{ a_0, a_1 \} \subset \llparenthesis b_1 b_0 \rrparenthesis$.
For $j \ge 1$, let $\alpha$ be the largest index $i$ such that $\{ a_{2i}, a_{2i + 1} \} \subset \llparenthesis b_{2j-1} b_{2j-2} \rrparenthesis$.
We have $a_{2\alpha+2} \in \llparenthesis y b_{2j-1} \rrparenthesis$ by the maximality of $\alpha$.
If $b_{2j} \in \llparenthesis a_{2\alpha} a_{2\alpha + 1} \rrparenthesis$, we have a linkable pair $(a_{2\alpha + 1}, b_{2j-1})$, a contradiction.
Hence, we have $b_{2j} \in \llparenthesis a_{2\alpha+1} b_{2j-2} \rrparenthesis$.
Therefore, we have $\{ a_{2\alpha}, a_{2\alpha + 1} \} \subset \llparenthesis b_{2j+1} b_{2j} \rrparenthesis$, as claimed. 
\end{proof}

Applying the claim for $j=\ell-1$, let $\beta$ be the largest index $i$ such that $\{ a_{2i}, a_{2i+1} \} \subset \llparenthesis b_{2\ell - 1} b_{2\ell - 2} \rrparenthesis$.
Since $a_{2\beta+1}$ and $b_{2\ell - 1}$ are not linkable, we have $b_{2\ell} \in \llparenthesis a_{2\beta+1} b_{2\ell-2} \rrparenthesis \subset \llparenthesis xy \rrparenthesis$, which contradicts that $b_{2\ell} \in \llparenthesis yx \rrparenthesis$.
\end{proof}

\begin{proof}[Proof of Proposition~{\upshape\ref{prop_yield_border_star}}]
Let $S_0 \in \mathcal{S}$ be a star which contains $uv$ as its edge.
Put $v_0 = u$ and $v_1 = v$.
Note that $S_0 \in \mathcal{S}_{v_1}$.
%

For a star $S \in \mathcal{S}_{v_1}$ with $E(S) = \{ u_0 u_1, u_1 u_2 , \ldots, u_{2\ell-1} u_{2\ell}, u_{2\ell} u_0 \}$, where $u_0=v_0$ and $u_1 = v_1$, let us assign a sequence of angles $\theta_i$ for $1 \le i \le 2\ell$ such that $\theta_i = \angle u_{i-1} u_i u_i^{+} - \angle u_{i-1} u_i u_{i+1}$ for $i$ odd, and $\theta_i = \angle u_{i-1} u_i u_{i+1}$ for $i$ even, where $u_{2\ell + 1} = u_0$.
We call $(\theta_1, \theta_2, \ldots, \theta_{2\ell})$ the angle sequence of $S$.   

Let $\mathcal{S}_1$ be a star decomposition reachable from $\mathcal{S}$ by a finite sequence of diagonal flips such that $\mathcal{S}_1$ contains a star $S_1 \in \mathcal{S}_{v_1}$, where the angle sequence of $S_1$ is as small as possible in the lexicographic order.

\begin{figure}[h]
\centering
\includegraphics[scale=0.3]{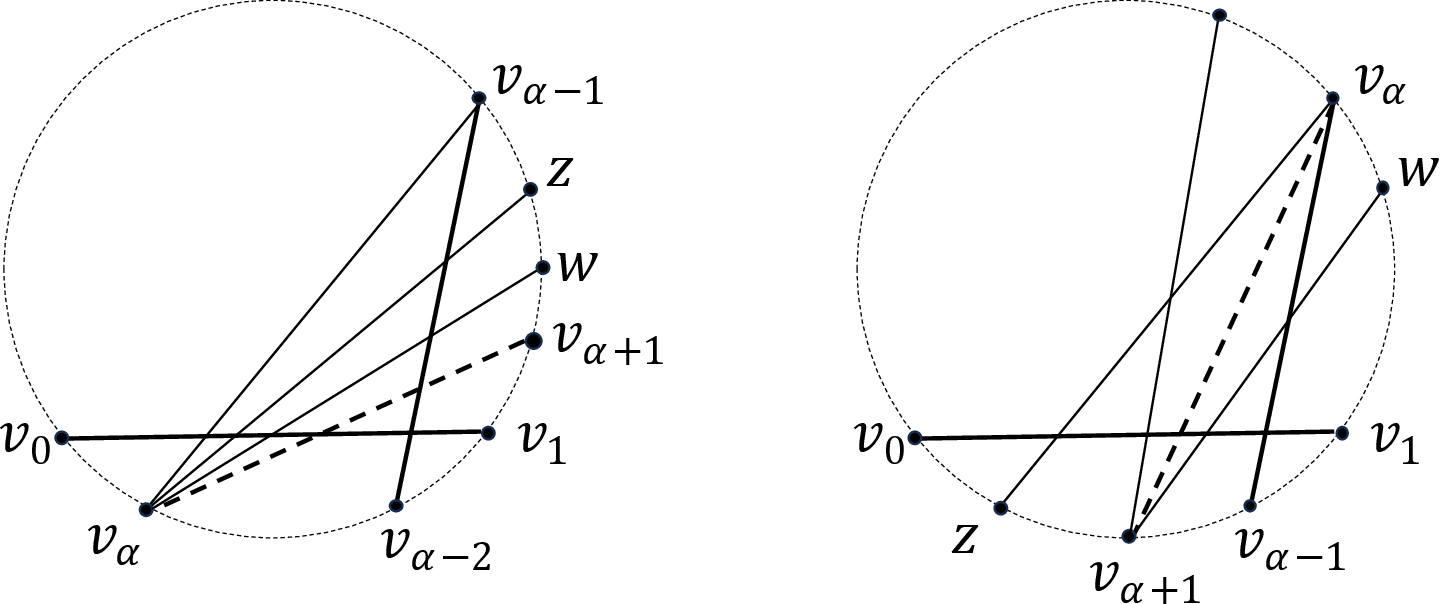}\\
\caption{Vertices and edges of the proof of Proposition \ref{prop_yield_border_star}
 in the case where $\alpha$ is odd(left) and $\alpha$ is even(right).
}
\label{fig_yield_border_star}
\end{figure}

It suffices to show that $S_1$ is the border star $S^{\ast}(P, v_1)$.
We assume that $S^{\ast}(P, v_1)$ is a $k$-star having $2k+1$ edges $v_0 v_1, v_1 v_2, \ldots, v_{2k-1} v_{2k}, v_{2k} v_0$.
Suppose to a contradiction that $S_1 \ne S^{\ast}(P, v_1)$.
Let us define $\alpha \in [2, 2k]$ such that $v_0 v_1, v_1 v_2, \ldots, v_{\alpha-1} v_{\alpha}$ are edges of $S_1$, and $v_{\alpha} v_{\alpha + 1} \not\in E(S_1)$, where $v_{2k + 1} = v_0$. 
Let $z$ be a vertex with $v_\alpha z \in E(S_1)$ (Fig. \ref{fig_yield_border_star}).

We consider two cases with respect to the parity of $\alpha$.
\medskip\\
\underline{Case $1$.} $\alpha$ is odd.

Since $v_{\alpha} v_{\alpha + 1} \in E(P)$, we have $z \in \llparenthesis v_{\alpha + 1} v_{\alpha - 1} \rrparenthesis$.

Let $S'$ be a star in $\mathcal{S}_1$ containing $zv_{\alpha}$ and $v_{\alpha} w$, where $w \in \llbracket v_{\alpha+1} z \rrparenthesis$.
By applying Lemma \ref{la_find_linkable_pair} with $S_1 = S_A$, $v_{\alpha-1} = a_0$, $v_{\alpha} = a_1$, $v_0 = x$, $v_1 = y$, and $S' = S_B$, $v_{\alpha} = b_0$, $w = b_1$, we have a linkable pair $(u, v) \in V(S_1) \times V(S')$ such that $u \in \llbracket v_0 v_{\alpha} \rrbracket$.
Then, we can flip a diagonal $\overline{v_{\alpha} z}$ to $\overline{uv}$, and one of resulting stars contains $v_0 v_1, v_1 v_2, \ldots, v_{\alpha} w$, a contradiction to the minimality of the angle sequence of $S_1$.
\medskip\\
\underline{Case $2$.} $\alpha$ is even.

Let $S'$ be a star in $\mathcal{S}_1$ having a linkable pair $(v_{\alpha + 1},  v_{\alpha}) \in V(S') \times V(S_1)$.
We may assume that $S'$ contains $v_{\alpha + 1}w$, where $w \in \llparenthesis v_1 v_{\alpha} \rrparenthesis$.
By applying Lemma \ref{la_find_linkable_pair} with $S_1 = S_A$, $v_{\alpha} = a_0$, $z = a_1$, $v_0 = x$, $v_1 = y$, and $S' = S_B$, $v_{\alpha + 1} = b_0$, $w = b_1$, we have a linkable pair $(u, v) \in V(S_1) \times V(S')$ such that $u \in \llbracket v_0 z \rrbracket$.
Then we have multiple linkable pairs for $S_1$ and $S'$, which contradicts  that $\mathcal{S}$ is maximal.
\end{proof}

Theorem \ref{thm_star_decomposition} follows from Proposition \ref{prop_yield_border_star}.

\begin{proof}[Proof of Theorem~{\upshape\ref{thm_star_decomposition}}]
We proceed by induction on the number of vertices.
Let $\mathcal{S}_i(V,E,D_i)$ for $i=1,2$.
By Proposition \ref{prop_yield_border_star}, we have a star decomposition $\mathcal{S}'_i$ for $i=1,2$ such that $\mathcal{S}_i$ is transformed into $\mathcal{S}'_i$ by a finite sequence of diagonal flips, and $\mathcal{S}'_i$ contains $S^{\ast}(P,v)$ for some vertex $v$.
Let us define a vertex set $V'$ and an edge set $E'$ as follows;\\
$E' = (E(P) \setminus E(S^{\ast}(P,v))) \cup \{ xy \,:\, yx \in E(S^{\ast}(P,v)) \setminus E(P) \}$, \\
$V' = V \setminus \{ v \in V \,:\, {\rm ~no~edge~of~}E'{\rm ~has~}v{\rm ~as~an~endvertex~} \}$.

Then we have that $V' = \emptyset$ and $E' = \emptyset$, or $P(V',E')$ is a cyclic polygon such that $P(V', E')$ admits a star decomposition $\mathcal{S}''_i = \mathcal{S}'_i \setminus \{ S^{\ast}(P,v) \}$ for $i=1,2$.
Since $v$ is not contained in $V'$, we have $|V'| < |V|$.
Hence, by inductive hypothesis, $\mathcal{S}''_1$ can be transformed into $\mathcal{S}''_2$, and this completes the proof. 
\end{proof}

Recall that $P^k_n$ denotes a cyclic polygon with $n$ vertices in which length of all the edges are $k$. 
As is mentioned in Section $1$, $P^k_n(V,E,D)$ admits a maximal $k$-star decomposition if and only if $D$ is $k$-saturated.  
Let $\mathcal{S}$ be a maximal star decomposition of $P_n^k$ with respect to a set of diagonals $D$.
By Theorem \ref{thm_star_decomposition}, $\mathcal{S}$ can be transformed from a specific $k$-star decomposition $\mathcal{S}_0$ with respect to $D_0$, where $D_0$ is $k$-saturated.
Hence, we have $D$ is also $k$-saturated and $\mathcal{S}$ is a $k$-star decomposition.

\begin{prop}\label{prop_k_triangulation}
All maximal star decompositions of $P^k_n$ are $k$-star decompositions.
\qed
\end{prop}

We remark that a double star with two $k$-stars may be transformed by a diagonal flip into another double star which contains an $\ell$-star with $\ell \ne k$ in general. (For example, see Fig. \ref{fig_diagonal_flip}.) 

\section{Further Discussions}

\subsection{Existence of a star decomposition}

\begin{figure}[h]
\centering
\includegraphics[scale=0.3]
{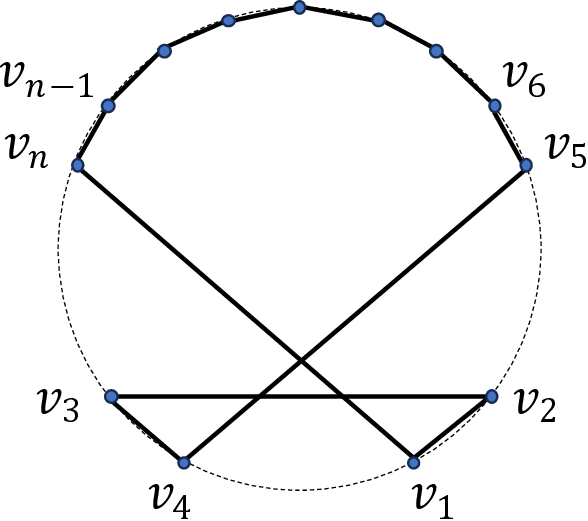}\\
\caption{A cyclic polygon which admits no star decomposition with $p-(n-2r)(n-2r-1)/2 = n-7 \ge 0$ for $n \ge 7$, where $r=2$.}
\label{fig_example_no_star_decomposition}
\end{figure}

If a cyclic polygon $P$ admits a star decomposition, by applying Theorem \ref{thm_num_diagonals}, we have $p - (n-2r)(n-2r-1)/2 \ge 0$, where $n$ is the number of vertices of $P$, $p$ is the number of linkable pairs of $P$, and $r$ is the rotation number of $P$, 
 because the left-hand side of the inequality equals the number of diagonals contained in a maximal star decomposition of $P$.
On the other hand, there exists a cyclic polygon having no star decomposition while satisfying $p - (n-2r)(n-2r-1)/2 \ge 0$.
(For example, see Fig. \ref{fig_example_no_star_decomposition}.)

Although whether a given cyclic polygon $P$ admits a star decomposition or not is determined algorithmically by trying to remove a border star one by one, it is desirable to clarify a necessary and sufficient condition for $P$ to admit a star decomposition.

\subsection{The number of star decompositions}

For a given cyclic polygon $P$, let $f(P)$ be the number of maximal star decompositions of $P$.
It is well-known that if $P$ is a convex $n$-gon, $f(P) = C_{n-2}$, where $C_n$ is the $n$-th Catalan number.
It is also known that $f(P_n^k) = \det (C_{n-i-j})_{1 \le i, j \le k}$ \cite{Jon05}, and there exists an explicit bijection between a family of star decompositions of $P_n^k$ and a family of $k$-fans of Dick paths of length $2(n-2k)$ \cite{SS2012}.   
It may be an interesting research topic to determine $f(P)$ for a given family of cyclic polygons.

\section*{Declarations}
The authors have no relevant financial or non-financial interests to disclose.




\begin{thebibliography}{99}

\bibitem{CP1992}
V. Capoyleas, J. Pach,
\newblock A Tur\'{a}n-type theorem on chords of a convex polygon,
\newblock \textsl{J. Combin. Theory Ser. B},
\newblock \textbf{56}\,(1992), 9--15.

\bibitem{DKM03}
A. Dress, J. H. Koolen, and V. Moulton,
\newblock On line arrangements in the hyperbolic plane,
\newblock \textsl{European J. Combin.}
\newblock \textbf{23}\,(2003), 549--557.

\bibitem{Jon05}
J. Jonsson,
\newblock Generalized triangulations and diagonal-free subsets of stack polyominoes,
\newblock \textsl{J. Combin. Theory Ser. A}, 
\newblock \textbf{112}\,(2005), 117--142.

\bibitem{Nak00}
T. Nakamigawa,
\newblock A generalization of diagonal flips in a convex polygon,
\newblock \textsl{Theoret. Comput. Sci.}
\newblock \textbf{235}\,(2000), 271-282.

\bibitem{PP2012}
V. Pilaud, M. Pocchiola,
\newblock Multitriangulations, pseudotriangulations and primitive sorting networks,
\newblock \textsl{Discrete Comput. Geom.}
\newblock \textbf{48}\,(2012), 142--191.

\bibitem{PS2009}
V. Pilaud, F. Santos,
\newblock Multitriangulations as complexes of star polygons,
\newblock \textsl{Discrete Comput. Geom.}
\newblock \textbf{41}\,(2009), 284--317.

\bibitem{PS2012}
V. Pilaud, F. Santos,
\newblock The brick polytope of a sorting network,
\newblock \textsl{European J. Combin.}
\newblock \textbf{33}\,(2012), 632--662.

\bibitem{PSZ2023}
V. Pilaud, F. Santos, and G. M. Ziegler,
\newblock Celebrating Loday's associahedron,
\newblock \textsl{Arch. Math.}
\newblock \textbf{121}\,(2023), 559--601.

\bibitem{PS2015}
V. Pilaud, C. Stump,
\newblock Brick polytopes of spherical subword complexes and generalized associahedra,
\newblock \textsl{Adv. Math.}
\newblock \textbf{276}\,(2015), 1--61.

\bibitem{SS2012}
L. Serrano and C. Stump,
\newblock Maximal Fillings of Moon Polyominoes, Simplicial Complexes, and Schubert Polynomials,
\newblock \textsl{Electron. J. Combin.}
\newblock \textbf{19}\,(2012). 

\bibitem{Stump2011}
C. Stump,
\newblock A new perspective on $k$-triangulations,
\newblock {J. Combin. Theory Ser. A.}
\newblock \textbf{118}\,(2011), 1794--1800.


\end{thebibliography}
\end{document}